\documentclass[a4paper,11pt]{scrartcl}

\usepackage{amsmath} 
\usepackage{amsfonts}
\usepackage{amssymb}
\usepackage{amsthm}

\usepackage[notcite,notref]{showkeys} 

\usepackage{isomath} 

\usepackage[T1]{fontenc} 
\usepackage[utf8]{inputenc}
\usepackage{lmodern}
\usepackage{microtype}
\usepackage{booktabs}

\usepackage{natbib}

\newtheorem{theorem}{Theorem} 
\newtheorem{lemma}[theorem]{Lemma}
\theoremstyle{remark}
\newtheorem{remark}[theorem]{Remark}

\usepackage{hyperref} 

\date{}
\author{Giacomo~Sbrana\footnote{Rouen Business School. 1, rue du Mar{\'
e}chal Juin, 76130 Mont-Saint-Aignan, France. E-mail \texttt{gsb@rouenbs.fr}} {} and Federico~Poloni\footnote{Dipartimento di Informatica, Universit\`a di Pisa. Largo Pontecorvo 3, 56127 Pisa, Italy. E-mail \texttt{fpoloni@di.unipi.it}}}
\title{A closed-form estimator for the multivariate GARCH(1,1) model}

\DeclareMathOperator{\vech}{vech}
\DeclareMathOperator{\vecc}{vec}
\renewcommand{\vec}{\vecc}
\DeclareMathOperator{\diag}{diag}
\DeclareMathOperator{\Cov}{Cov}

\DeclareMathOperator{\im}{im}

\newcommand{\norm}[1]{\left\Vert #1\right\Vert}
\newcommand{\abs}[1]{\left\vert #1\right\vert}
\newcommand{\m}[1]{\begin{bmatrix} #1 \end{bmatrix}}

\begin{document}
\maketitle
\begin{abstract}
We provide a closed-form estimator based on the VARMA representation for the unrestricted multivariate GARCH(1,1). We show that all parameters can be derived using basic linear algebra tools. We show that the estimator is consistent and asymptotically normal distributed. Our results allow also to derive a closed form for the parameters in the context of temporal aggregation of multivariate GARCH(1,1) by solving the equations as in \cite{H2}.

\end{abstract}

\section{Introduction}
Estimating a multivariate GARCH(1,1) model is a challenging task. The most common tool for this purpose is the quasi maximum likelihood (QML) estimator which requires rather sophisticated optimization techniques. 
In this paper we present a simple and fast method of moments which makes the estimation of the multivariate GARCH(1,1) model more accessible. Our results represent the multivariate generalization of the analytical results already achieved by \cite{LK} for the scalar case.

Our estimator is consistent and, under additional assumptions on the moments, asymptotically normal distributed. Due to the difficulties in estimating multivariate GARCH(1,1) models our estimator may then be used to provide a consistent initial estimate when implementing numerical optimization techniques for the QML estimation. This is especially true when large-scale models are employed.

Several restricted models have been proposed by the previous literature in order to reduce the number of parameters, such as Diagonal VEC (\cite{BEW}), BEKK-GARCH (\cite{EK}), CCC-GARCH (\cite{B}). Interestingly, our results are valid in general. Therefore in the framework we stick to the unrestricted multivariate GARCH(1,1).

Finally, our results extend the results of \cite{H2} in the context of temporal aggregation of multivariate GARCH(1,1). Indeed, our results allow to derive the parameters of the temporally aggregated GARCH for any aggregation frequency. In other words, given the parameters of the disaggregated process, those of the aggregate one are analytical functions of the disaggregate parameters. Alternatively, one can also use the moments of the disaggregated GARCH to produce an initial estimate of the parameters of temporally aggregated processes. The former estimator is again consistent and asymptotically normal when some moments conditions hold.

\section{Framework}

Consider the following unrestricted multivariate GARCH(1,1) model
\begin{align*}
y_t=& H_t^{1/2}  \epsilon_t, \quad t=1,2,\dots,n,
\end{align*}
where $y_t$ is a $d$-dimensional zero-mean, serially uncorrelated process. In addition, we have that $\epsilon_t\in\mathbb{R}^{d\times 1}$ is an i.i.d. white noise vector with zero mean and variance $I_d$.
Moreover, the conditional covariance matrix is given by
\begin{align}\label{garchrec}
\vech{(H_t)}=& c+A\vech(y_{t-1} y_{t-1}^T)+B\vech(H_{t-1}), \quad t=2,3,\dots,n,
\end{align}
where $\vech(M)$ represents the operator that stacks the elements of the lower triangular part of a symmetric matrix $M$ to form a $\bar{d}\times 1$ vector, with $\bar{d}=\frac{d(d+1)}{2}$.

In what follows we make the following assumptions:
\begin{enumerate}
\item $H_t$ is positive definite almost surely for each $t$.
\item All eigenvalues of the matrix $A+B$ have modulus smaller than one. 
\item The process $y_t$ is ergodic, $\beta$-mixing, and strictly stationary. 
\item \label{moments} The fourth moments of $y_t$ exist and are finite.
\end{enumerate}

\cite{Bou06} provides sufficient conditions that ensure a strictly stationary, ergodic and $\beta$-mixing solution of the vector GARCH process (these can also be found in \cite[Theorem~11.5]{FZ}).

The following stronger assumption is used only in some central limit results:
\begin{enumerate}
 \item[5] The eighth moments of $y_t$ exist and are finite.
\end{enumerate}
When the distribution of $\epsilon_t$ is spherical, \cite[Theorem~3]{H1} has given an algebraic condition equivalent to Assumption~\ref{moments} that is easy to test in practice. However, we do not need to assume sphericity here.

The VARMA(1,1) representation of the multivariate GARCH(1,1) is obtained by defining $h_t\equiv\vech(H_t)$, $x_t\equiv \vech(y_{t} y_{t}^T)$ and $\xi_{t}\equiv x_{t}-h_t$; the recurrence relation \eqref{garchrec} is equivalent to
\begin{align*}
x_{t}= c + \Phi x_{t-1} + \xi_{t} - B \xi_{t-1}.
\end{align*}
By eliminating $x_{t-1}$ recursively, we find that asymptotically the following formula holds
\begin{align*}
x_{t}= h + \sum_{i=0}^{\infty}\Theta_{i} \xi_{t-i},
\end{align*}
where $\Phi=A+B$; $h=\vech(H)=(I-\Phi)^{-1}c$;  $\Theta_{0}=I_{d}$ and $\Theta_{i}=(A+B)^{i-1}A$ for $i \geq 1$.
The interest of this formulation lies in the fact that $\xi_t$ is a martingale difference sequence. We define $\Sigma \equiv E[\xi_{t}\xi_{t}^{T}]$; note that $E[\xi_{t}\xi_{t+s}^{T}]=0$ for $s \geq 1$.

\section{Closed form Estimation}\label{sec:closedForm}
Under Assumptions 1--4, the autocovariances of $x_t$ exist and are finite, and they are given by
\[
 M_{k} =E[(x_{t+k}-h)(x_{t}-h)^{T}] =\sum_{i=0}^{\infty}\Theta_{i+k}\Sigma \Theta_{i}^T.
\]
From the VARMA(1,1) representation using the standard Yule-Walker results we have that
\begin{displaymath}
M_{k+1} =\Phi M_{k}, \quad \text{for all $k\geq1$},
\end{displaymath}
thus $\Phi$ can be obtained analytically as
\begin{equation}\label{Phi}
\Phi=M_{k+1}M_{k}^{-1},\quad \text{for all $k\geq1$.}
\end{equation}

These results are well known and can also be found for example in the book of \cite{R} as well as in \cite{H1} page 32 (for the univariate case see \cite{LK}). Therefore, we can estimate $\hat{\Phi}=\hat{M}_{2} \hat{M}_{1}^{-1}$. Consider now the first-order moving average vector 
\[
j_{t}=x_{t}-\Phi x_{t-1} =  c + \xi_{t} - B \xi_{t-1}.
\]
The autocovariances of $j_t$ are
\begin{equation}\label{Gammas}
\begin{aligned}
\Gamma_0 &\equiv E[(j_{t}-c)(j_{t}-c)^{T}]= \Sigma+B\Sigma B^T=M_{0}-M_{1}\Phi^{T}-\Phi M_{1}^{T}+\Phi M_{0}\Phi^{T},\\
\Gamma_1 &\equiv E[(j_{t}-c)(j_{t-1}-c)^{T}]=-B \Sigma=M_{1}-M_{2}\Phi^{T}-\Phi M_{0}+\Phi M_{1}\Phi^{T}=M_1-\Phi M_0.
\end{aligned}
\end{equation}

We can combine the former two equations with simple manipulations to derive two separate equations for $B$ and $\Sigma$
\begin{gather}
 \Gamma_1^T+\Gamma_0 B^T +\Gamma_1 (B^T)^2=0, \label{pme}\\
 \Gamma_0=\Sigma+\Gamma_1\Sigma^{-1}\Gamma_1^T. \label{nme}
\end{gather}



In the scalar case, \eqref{pme} is a quadratic equation; the approach suggested by \cite{LK} consists essentially in deriving an estimator by solving this equation.

This method, however, need not be restricted to the univariate GARCH. In the multivariate case, basic linear algebra techniques can be used to derive a closed form in terms of eigenvalues and eigenvectors. We present them in the next section.

\subsection{Closed formula for $B$}
The following procedure can be used to obtain $B$ as analytical function of $\Gamma_0$ and $\Gamma_1$.
\begin{enumerate}
 \item Form the $2\bar{d} \times 2\bar{d}$ matrix
\begin{align}\label{defP}
P=
 \begin{bmatrix}
 0 & I\\
 -\Gamma_1^{-1}\Gamma_1^T & -\Gamma_1^{-1}\Gamma_0
\end{bmatrix}.
\end{align}
\item One can prove (Lemma~\ref{pairinglemma} in the following) that the eigenvalues of $P$ come in pairs $(\lambda,1/\lambda)$. Therefore, unless there are eigenvalues that lie exactly on the unit circle, half of the $2\bar{d}$ eigenvalues satisfy $\abs{\lambda}<1$, and we may reorder them so that $\abs{\lambda_i} < 1$ for $i=1,2,\dots,\bar{d}$. Moreover, consider the associated eigenvectors $w_i$ and partition them as
\[
 w_i=\m{u_i\\v_i}, \quad u_i,v_i\in\mathbb{R}^{\bar{d}}.
\]
\item Now, a solution to the matrix equation \eqref{pme} is given by
\begin{equation}\label{B}
B=(U^T)^{-1} D U^T = \m{u_1 & u_2 & \cdots & u_{\bar{d}}}^{-T} \m{\lambda_1 \\ & \lambda_2 \\ & & \ddots \\ & & & \lambda_{\bar{d}}} \m{u_1 & u_2 & \cdots & u_{\bar{d}}}^T.
\end{equation}
\end{enumerate}
In Section~\ref{sec:linearAlgebra} and Appendix~\ref{secproofs}, we recall the theoretical results in linear algebra that ensure the functioning of this procedure.

\subsection{Estimation using the closed formula}
Using this formula, an estimation procedure can be derived as following:
\begin{enumerate}
 \item Given the data, compute the observed average and the first three autocovariances of $x_t$:
 \begin{align*}
  \hat{h}&=\frac{1}{n}\sum_{t=1}^{n}x_{t}, & \hat{M}_0&=\frac{1}{n}\sum_{t=1}^{n}[(x_{t}-\hat{h})(x_{t}-\hat{h})^T],\\
  \hat{M}_1&=\frac{1}{(n-1)}\sum_{t=1}^{n-1}[(x_{t+1}-\hat{h})(x_{t}-\hat{h})^T], & \hat{M}_2&=\frac{1}{(n-2)}\sum_{t=1}^{n-2}[(x_{t+2}-\hat{h})(x_{t}-\hat{h})^T].
 \end{align*}
\item Evaluate $\hat{\Phi}=\hat{M}_2 \hat{M}_1^{-1}$, $\hat{\Gamma}_0= \hat{M}_{0}-\hat{M}_{1}\hat{\Phi}^{T}-\hat{\Phi} \hat{M}_{1}^{T}+\hat{\Phi} \hat{M}_{0}\hat{\Phi}^{T},$ and $\hat{\Gamma}_1=\hat{M}_1-\hat{\Phi} \hat{M}_0$, as provided by \eqref{Gammas}.
 \item Use the above procedure based on eigenvalue computation to get an estimated $\hat{B}$.
 \item Finally, recover the other two parameters as $\hat{A}=\hat{\Phi}-\hat{B}$ and $\hat{c}=(I-\hat{A}-\hat{B})\hat{h}$.
\end{enumerate}

\subsection{Asymptotic properties}
In this section and the following, the symbols $\stackrel{\text{p}}{\rightarrow}$ and $\stackrel{\text{dist}}{\rightarrow}$ stand for convergence in probability and distribution, respectively.

The consistency of the QML estimator has been shown by \cite{J} while \cite{CL} provides the asymptotic normality of the QML estimator in the context of BEKK formulation. However, as noted in \cite{BLR}: \emph{The asymptotic properties of ML and QML estimators in multivariate GARCH models are not yet firmly established, and are difficult to derive from low level assumptions [...] Asymptotic normality of the QMLE is not established generally. [...] Researchers who use MGARCH models have generally proceeded as if asymptotic normality holds in all cases.}
Here we provide the asymptotic properties of our closed-form estimator $\Lambda=(c,\vec(A),\vec(B))$, which is function of the moments of $x_t$ only, that is, $\Lambda=F(h , M_{0} , M_{1}, M_{2})$. 

\begin{lemma}\label{Convergence}
Let 
\[
\mathfrak{m} = \m{h\\\vec M_0 \\ \vec M_1 \\ \vec M_2}. 
\]
Under Assumptions~1--4, $\hat{\mathfrak{m}} \stackrel{\text{p}}{\rightarrow}\mathfrak{m}$. In addition, if Assumption~5 holds, we have
\begin{align}
\sqrt{n}(\hat{\mathfrak{m}}-\mathfrak{m}) \stackrel{\text{dist}}{\rightarrow} N(0,\Psi),
\end{align}
where $\Psi=\Cov[\mathfrak{m}]$.
\end{lemma}
\begin{proof}
Under these assumptions, both $h$ and $M_k$ are finite, in addition $\hat{h}$ and $\hat{M}_k$ are consistent given the law of large numbers for stationary and ergodic processes. Finally, the joint distribution converges to a normal distribution thanks to the standard central limit theory for strongly mixing sequences.
\end{proof}
Note that $\mathfrak{m}$ is consistent even if Assumption~5 does not hold; in this case, however, $\hat{\mathfrak{m}}$ converges with a slower rate (see the discussion in \cite{LK}). Moreover, the limit joint distribution of the moments is not Gaussian.

If the noise distribution is spherical, then, as noted by \cite{H1}, the cross-covariance between $h$ and $\vec(M_k)$ vanishes for each $k$, as it is an odd-power function of the noise $\epsilon_t$. In this case, the explicit expression for $\Psi$ is
\begin{displaymath}
\Psi=
\m{ \Cov[h] & 0_{\bar{d}\times 3\bar{d}^2}\\
 0_{3\bar{d}^2\times \bar{d} } & \Cov[M_k,M_l]}
\end{displaymath}
where:
\begin{displaymath}
\Cov[h]=M_0+\sum_{k=1}^{\infty}M_k+\sum_{k=1}^{\infty}M_k^T
\end{displaymath}
and
\begin{displaymath}
\Cov[M_k,M_l]=E\left[\left(\vec{[(x_{t+k}-h)(x_{t}-h)^T]}-\vec{M_k}\right)\left(\vec{[(x_{t+l}-h)(x_{t}-h)^T]}-\vec{M_l}\right)^T\right]
\end{displaymath}
for $k,l=0,1,2$.

The following theorem represents a direct consequence of Lemma \ref{Convergence}.
\begin{theorem}\label{Normality}
Under Assumptions 1--4, the estimator $\hat{\Lambda}=(\hat{c},\vec(\hat{A}),\vec(\hat{B}))=F(\mathfrak{m})$ is consistent,  i.e. $\hat{\Lambda} \stackrel{p}{\rightarrow} \Lambda$. In addition, if Assumption~5 holds,
$\sqrt{n}(\hat{\Lambda} -\Lambda){\rightarrow}N(0,\Xi)$ with
\begin{equation}\label{usesPartialDerivatives}
\Xi=\left(\frac{\partial F(\mathfrak{m})}{\partial \mathfrak{m}}\right)\Psi\left(\frac{\partial F(\mathfrak{m})}{\partial \mathfrak{m}}\right)^{T}.
\end{equation}
\end{theorem}
\begin{proof}
The theorem follows from the continuous mapping theorem; in addition, the explicit expression for the covariance is obtained using the delta method. It remains to show that the partial derivative $\frac{\partial F(\mathfrak{m})}{\partial \mathfrak{m}}$ exists, which is proven in the end of Section~\ref{sec:linearAlgebra}.
\end{proof}
Note that the asymptotic properties of the closed form estimator might be employed to prove these of the QML in general (as discussed by \cite{BLR}). However, we leave this for future research.
\section{Temporal aggregation}
An interesting consequence of the results above is the direct extension to the temporal aggregation. In fact, we can now derive (as well as estimate) the parameters of the temporally aggregated multivariate GARCH(1,1) as discussed in \cite{H2}.

Temporal aggregation of a GARCH can be conducted in two different forms, depending on whether we are interested in \emph{stock} or \emph{flow} variables. We are interested in deriving a GARCH representation for the process $y^{(m)}$ aggregated over $m$ periods, which is defined in the two cases as
\begin{align*}
 y_{mt}^{(m)} &= y_{mt}, & &\text{(stock variables)}\\
 y_{mt}^{(m)} &= y_{mt} + y_{mt-1} \dotso + y_{mt-m+1} + w_{mt}^{(m)}. &&\text{(flow variables)}
\end{align*}
We denote $x_{mt}^{(m)}=\vech(y_{mt}^{(m)} y_{mt}^{(m)T})$. While in the stock case $x_{mt}^{(m)} = x_{mt}$, the relation between the second moments is more involved for flow variables.

\cite{H2} shows that the temporally aggregated process follows a weak VARMA(1,1)
\begin{align}
x_{mt}^{(m)}= c^{(m)} + \Phi ^{(m)} x_{m(t-1)}^{(m)} + \xi_{mt}^{(m)} - B ^{(m)} \xi_{m(t-1)}^{(m)}.
\end{align}
Then, equations for $B^{(m)}$ formally analogous to \eqref{Gammas} are derived:
\begin{align*}
 \Gamma_0^{(m)}&=\Sigma^{(m)} + B^{(m)} \Sigma^{(m)} (B^{(m)})^T,\\
 \Gamma_1^{(m)} &= - B^{(m)} \Sigma^{(m)},
\end{align*}
with
\begin{align*}
      \Gamma_0^{(m)}&=\sum_{i=0}^{m}J_{i}^{s}\Sigma J_{i}^{sT},\\
      \Gamma_v^{(m)}&=J_{m}^{s}\Sigma,\\
      J_{0}^{s} &=I_{\bar{d}},\\
      J_{i}^{s} &=\Phi^{i-1}A,\\
      J_{m}^{s} &=-\Phi^{m-1}B,\\
\end{align*}
for the case of stock variables, and
\begin{align*}
 \Gamma_0^{(m)}&=\sum_{i=0}^{2m-1}J_{i}^{f}\Sigma J_{i}^{fT}+\Sigma_{w}^{m}+\Phi^{m}\Sigma_{w}^{m}(\Phi^{T})^{m},\\
 \Gamma_1^{(m)}&=\sum_{i=0}^{m-1}J_{i+m}^{f}\Sigma J_{i}^{fT}-\Phi^{m}\Sigma_{w}^{m},\\
 J_{0}^{f} &=I_{\bar{d}},\\
      J_{i}^{f} &=I_{\bar{d}}+A+\Phi A+\cdots+\Phi^{i-1} A,\;\;\;\;i=1,\cdots,m-1\\
      J_{i}^{f} &=[I_{\bar{d}}+\Phi+\cdots+\Phi^{m-2}]A-\Phi^{m-1}B,\;\;\;\;i=m\\
      J_{i}^{f} &=[\Phi^{i-m}+\Phi^{i-m+1}+\cdots+\Phi^{m-2}]A-\Phi^{m-1}B,\;\;\;\;i=m+1,\cdots,2m-2,\\
      J_{2m-1}^{f} &=-\Phi^{m-1}B, \\
\end{align*}
for the case of flow variables (where an explicit expression for $\Sigma_{w}^{m}$ can be found in \cite{H2} eq.19). The author, however, rearranges them to eliminate $\Sigma^{(m)}$ in a form that differs slightly from our \eqref{pme}, and for which deriving an explicit solution is more complicated. In his words, \emph{``As of $B^{(m)}$, (29) is a system of nonlinear equations that cannot be solved explicitly.''}. With the tools provided in this paper, an explicit solution is now available. Equation (29) in \cite{H2} can be replaced with \eqref{pme} in this paper. Therefore using
\begin{align}\label{defPm}
P^{(m)}=
 \begin{bmatrix}
 0 & I\\
 -\Gamma_1^{(m)-1}\Gamma_1^{(m)T} & -\Gamma_1^{(m)-1}\Gamma_0^{(m)},
\end{bmatrix}.
\end{align}
we can carry on the procedure described in Section~\ref{sec:closedForm}. Again, the eigenvalues to choose are those inside the unit circle.
%

\cite{H2} shows that $\Gamma_0^{(m)}$ and $\Gamma_1^{(m)}$ are analytical functions of $A$, $B$ and $E(\xi_t \xi_t^T)$ and these are are function of the moments of $x_t$ only. It turns out that a closed form estimator of the temporally aggregated GARCH(1,1) can be derived as an analytical function $\Lambda^{(m)}=(c^{(m)},\vec A^{(m)},\vec B^{(m)})=G^{(m)}(h , \vec M_{0} ,\vec  M_{1},\vec M_{2})$. In particular, we can use for the estimation of the aggregated GARCH(1,1) the estimated moments of the high-frequency data, for which more information is available.

The estimator enjoys the same asymptotic properties.
\begin{theorem}\label{NormalityAggregate}
Under Assumptions 1--4, the estimator $\hat{\Lambda}^{(m)}=(\hat{c}^{(m)},\vec \hat{A^{(m)}},\vec \hat{B^{(m)}})=G^{(m)}(\mathfrak{m})$ is consistent, i.e. $\hat{\Lambda}^{(m)} \stackrel{p}{\rightarrow} \Lambda^{(m)}$. In addition, if Assumption~5 holds as well, we have:
$\sqrt{n}(\hat{\Lambda}^{(m)} -\Lambda){\rightarrow}N(0,\Xi^{(m)})$, with
\begin{equation}\label{usesPartialDerivatives2}
\Xi^{(m)}=\left(\frac{\partial G^{(m)}(\mathfrak{m})}{\partial \mathfrak{m}}\right)\Psi\left(\frac{\partial G^{(m)}(\mathfrak{m})}{\partial \mathfrak{m}}\right)^{T}.
\end{equation}
\end{theorem} 

\cite{HR} discuss the estimation of temporally aggregated multivariate GARCH(1,1). Our estimator can be employed as a simple estimator when the number of observations is sufficiently high. Alternatively, it can be used as a consistent starting value for the QML estimation.

\section{Palindromic matrix equations and eigenvalue problems}\label{sec:linearAlgebra}
In this section, we present the linear algebra results that lead to the estimator of $B$ (as well as $B^{(m)}$). These matrix equations have been studied by many authors in linear algebra literature, see e.g. \cite{GohLR}, \cite{EngRR93}, \cite{Mei02} and the references therein; we focus here on providing a non-technical exposition. Proofs of some of the results are presented in Appendix~\ref{secproofs}.

The problem of computing one or more pairs $(\lambda,u)$ satisfying
\begin{align}\label{pep}
  (\lambda^2 \Gamma_1 +\lambda\Gamma_0 +\Gamma_1^T)u=&0, & u \neq &0,
\end{align}
is known as \emph{palindromic quadratic eigenvalue problem} \cite{Vibes}. The complex numbers $\lambda$ are called \emph{generalized eigenvalues} and the vectors $u$ \emph{generalized eigenvectors}. It is indeed a generalization of the standard eigenvalue problem, i.e., given a matrix $A$ finding pairs $(\lambda,u)$ satisfying $(A-\lambda I)u=0$.

First, we show that all the solutions to \eqref{pme} can be constructed from generalized eigenvalues and eigenvectors of \eqref{pep}.
\begin{lemma}\label{solGenEig}
 Let $(\lambda_1,u_1), (\lambda_2,u_2), \dots, (\lambda_{\bar{d}}, u_{\bar{d}})$ be $\bar{d}$ different pairs of generalized eigenvalues and eigenvectors of the problem \eqref{pep}, such that the matrix $U$ (as in \eqref{B}) is invertible. Then, $B=(U^T)^{-1} D U^T$ is a solution of \eqref{pme}. All the solutions of \eqref{pme} can be obtained in this way.
\end{lemma}
Moreover, in our problem the possible values of $\lambda$ can be ``paired''.
\begin{lemma}\label{pairinglemma}
Let $\Gamma_0$ and $\Gamma_1$ be real matrices, with $\Gamma_0$ symmetric.
 If $\lambda \neq 0$ is a generalized eigenvalue of \eqref{pep}, then $1/\lambda$ is one as well.
\end{lemma}
Finally, the following result shows that we can reduce the palindromic eigenvalue problem to a standard eigenvalue problem.
\begin{lemma}\label{eigenlemma} Suppose that the matrix $\Gamma_1$ is invertible.
\begin{enumerate}
 \item Let $(\lambda,u)$ be a solution of \eqref{pep}. Then, $\lambda$ is an eigenvalue of $P$,
with corresponding eigenvector
\[
w=
\begin{bmatrix}
 u\\ \lambda u
\end{bmatrix}
\]
\item Conversely, if $\lambda$ is an eigenvalue of $P$ with eigenvector $w$, then $(\lambda,u)$ satisfy \eqref{pep}, where $u$ is the vector formed by the first $\bar{d}$ entries of $w$.
\end{enumerate}
\end{lemma}

By combining all the above results, we can obtain a closed form for $B$. Note that different solutions are possible; namely, every choice of $\bar{d}$ eigenvalues out of the $2\bar{d}$ of $P$ gives a different $B$ satisfying \eqref{pme}; however, only the one with $\abs{\lambda_i}<1$, $i=1,2,\dots,\bar{d}$ results in a $B$ with all its roots inside the unit circle.

\begin{remark}
The invertibility of $\Gamma_1$ is not a crucial assumption. If $\Gamma_1$ is singular, we can obtain in a similar way not an eigenvalue problem of the form $Mv=\lambda v$, but the slightly more general form $Mv=\lambda Nv$. This is known as \emph{generalized eigenvalue problem}, and there are plenty of algorithms to find a closed-form solution to it. For instance, the Matlab command \texttt{eig(M,N)}. Similarly, if $P$ is not diagonalizable, solutions $B$ to the matrix equation \eqref{pme} can be defined in terms of its Jordan canonical form.
\end{remark}

The existence of the partial derivatives $\frac{\partial F(h , \vec M_{0} , \vec M_{1}, \vec M_{2})}{\partial (h , \vec M_{0} , \vec M_{1}, \vec M_{2})}$ that are needed in Theorem~\ref{Normality} can be shown, again under the condition that $P$ has no unimodular eigenvalues.
\begin{lemma}\label{analytical}
 Suppose that the matrix $P$ has no eigenvalues on the unit circle. Then, $B$ is an analytical function of the equation coefficients $\Gamma_0$, $\Gamma_1$.
\end{lemma}
Since $\Gamma_0$ and $\Gamma_1$ are in turn analytical functions of the $M_i$, the partial derivative exists. A sketch of proof of this result is in the appendix, together with a more explicit expression for the Jacobian.

\section{Small sample issues}\label{unimodular}
The results provided in this paper should be employed with caution whenever the sample size $n$ is not large enough. The closed-form estimator is based on the sample estimates of $\hat{M}_0$, $\hat{M}_1$, $\hat{M}_2$. However, it may be the case that the sample moments do not respect all the stated assumptions.
More specifically, three different kind of issues can arise:
\begin{description}
 \item[Positivity] $\hat{c}$, $\hat{A}$ and $\hat{B}$ do not guarantee that $H_t$ is positive definite.
 \item[Stationarity] The roots of $\hat{\Phi}=\hat{M}_{k+1}\hat{M}_{k}^{-1}$ lie on or outside the unit circle.
 \item[Invertibility] $\hat{B}$ has unimodular eigenvalues (i.e., on the unit circle).
 \end{description}
When the GARCH parameters are estimated via maximum likelihood, the constraints of respecting these conditions are usually imposed when solving the optimization problem; see for instance \cite{CO}. Since black-box optimization routines are used, additional constraint are easy to impose, but they make the resulting problem more complicated to solve. On the plus side, they guarantee that the 
resulting model has the desired properties, provided that the iterative optimization procedure does not fail. Instead, with an exact moment-based estimator, if one or more of these conditions fail, then the best way out is modifying the sample moments or the estimates \emph{a posteriori} to make sure that they satisfy these constraints. We discuss briefly these problems that may arise when our closed form estimation is employed.

\subsection{Positivity}
Sufficient conditions for positivity are discussed by several authors (see \cite{G}, \cite{CO}, \cite{FZ}). However, as far as we know, the problem of finding necessary conditions has not been dealt with in literature. Indeed, even the simpler problem of finding all linear maps among symmetric matrix spaces that preserve positive semi-definiteness has no simple solution, see for instance \cite[Chapters~2 and 3]{Bhatia}. Our estimation procedure does not always produce an estimated GARCH satisfying the sufficient conditions cited above. 
We do not deal here with the problem of finding a weaker set of conditions that can be preserved.

\subsection{Stationarity}
In small samples the estimate of $\hat{\Phi}=\hat{M}_{k+1}\hat{M}_{k}^{-1}$ can have eigenvalues on or outside the unit circle; moreover, the values $\hat{\Phi}^{(k)}$ computed by choosing different values of $k$ in the former expression will in general be different. The choice described above of taking $\hat{\Phi}^{(1)}$ and ignoring all the other autocovariances ratios is the simplest way out of the latter problem. \cite{LK} discuss this problem in the scalar case, and suggest as another valid approach taking $\frac13\left(\hat{\Phi}^{(1)} + \hat{\Phi}^{(2)} + \hat{\Phi}^{(3)}\right)$, or in general any convex combination $\sum w_i \Phi^{(i)}$.

To avoid problems with noninvertibility or outliers, in the multivariate case it is more advisable use instead a least-square solution $\hat{\Phi}_*$ of the system
\[
 \hat{\Phi}_* \m{w_1\hat{M}_1 & w_2\hat{M}_2 & \cdots & w_{n-1}\hat{M}_{n-1}}  =  \m{w_1\hat{M}_2 & w_2\hat{M}_3 & \cdots & w_{n-1}\hat{M}_{n}},
\]
again with suitably-chosen weights $w_i$. None of these solutions (and no choice of weights) clearly stands out. In particular, all of them may result in estimates $\hat{\Phi}$ with eigenvalues equal or larger than 1. When this happens, a simple fix is projecting the estimate on the space of acceptable GARCH solutions by altering the eigenvalues that lie on or outside the unit circle. 

\subsection{Invertibility}
Our assumptions on the solution $B$ guarantee that it has all eigenvalues inside the unit disc, and thus that the matrix $P$ has no unimodular eigenvalues. However, once again the sample autocovariances from a finite-time realization of a GARCH process may lead to sample estimates of $\hat{\Gamma}_0$ and $\hat{\Gamma}_1$ that do no necessarily guarantee that $\hat{P}$ has no eigenvalues on the unit circle. This is especially true when the process is close to a non-invertible one. When $\hat{B}$ has unimodular eigenvalues the invertibility condition does not hold. In addition the following theorem sheds light on the consequences of having eigenvalues lying on the unit disk.
\begin{theorem}
 \begin{enumerate}
The following results hold.
  \item If $P$ has no eigenvalues lying on the unit circle, then there exist unique solutions $B$ and $\Sigma$, where $\Sigma=\Sigma^T$ and $\abs{\lambda}<1$ for each eigenvalue $\lambda$ of $B$, and they can be computed with the above procedure.
  \item If $P$ has eigenvalues lying on the unit circle, then $B$ and a positive definite $\Sigma$ satisfying \eqref{nme} exist only if some strong additional conditions are satisfied (in particular, all unimodular eigenvalues should have even multiplicity). In this case, $B$ always has unimodular eigenvalues.
 \end{enumerate}
\end{theorem}
A full proof is more technical than those for the other linear algebra results that we reported; we omit it and refer to \cite{EngRR93} for a complete presentation. However, the last assertion is clear in view of our derivation: since the eigenvalues of $B$ are a subset of those of $P$, $B$ cannot have all its eigenvalues inside the unit circle if $P$ has less than $\bar{d}$ eigenvalues in that domain. If the existence conditions are not satisfied, we can still compute solutions $B$ with $\rho(B) = 1$ with the procedure of Section~\ref{sec:closedForm}; there are multiple solutions, according to which unimodular eigenvalues we choose, but none of them will result in a symmetric $\Sigma$. Ad-hoc modifications of $P$ can be made when unimodular eigenvalues are detected, but in general the accuracy of the computed solution is expected to decrease. Indeed, we show in Appendix~\ref{appDer} that the derivative of $B$ with respect to the moments can become unbounded when it has eigenvalues equal to $1$. 

We are currently working on developing a general procedure for computing a small-norm modification of the $\hat{M}_k$ that makes the estimated model invertible, rooted on results in linear algebra and eigenvalue perturbation theory.

We point out that the same problem arises in the scalar case treated by \cite{LK}: when (in our notation) $\Gamma_0/\Gamma_1>2$, the scalar quadratic equation \eqref{pme} has two complex conjugate solutions with modulus 1, and the procedure breaks down.






\section{Acknowledgments}
We are grateful to Christian Francq, Mark M. Meerschaert, Hans-Peter Scheffler, Dalibor Volny and Jean-Michel Zako\"ian for providing interesting and useful comments on an early revision of this paper. All remaining errors in the paper are solely the responsibility of the authors.

F.~Poloni thanks the Rouen Business School for support for a brief research visit, which was instrumental in completing the paper.

\appendix
\section{Proofs}\label{secproofs}

\subsection*{Proof  of Lemma~\ref{solGenEig}}
\begin{proof}
 Using the fact that the $(\lambda_i,u_i)$ are generalized eigenvalues, one can check directly that each column of the matrix
 \[
  \Gamma_1 UD^2 + \Gamma_0 UD + \Gamma_1^T U
 \]
 is zero; therefore,
 \[
 0=(\Gamma_1 UD^2 + \Gamma_0 UD + \Gamma_1^T U) U^{-1} = \Gamma (UDU^{-1})(UDU^{-1}) + \Gamma_0 (UDU^{-1}) + \Gamma_1^T,
 \]
 as required. For the converse implication, let $B^T=(UDU^{-1})$ be the spectral decomposition of a solution; we can reverse all the steps and obtain that each $(\lambda_i,u_i)$ is a generalized eigenpair.
\end{proof}
\subsection*{Proof of Lemma~\ref{pairinglemma}}
\begin{proof}
 Let $\lambda$ satisfy \eqref{pep} for some choice of $u\neq 0$. Since $\Gamma_0$, $\Gamma_1$ are real, we can take the complex conjugate of every term and get
 \[
  (\bar{\lambda}^2 \Gamma_1 + \bar{\lambda}\Gamma_0+\Gamma_1^T) \bar{u}=0,
 \]
 where $\bar{\lambda}$ and $\bar{u}$ denote (componentwise) complex conjugation. In particular, this implies that
 \[
  \det (\bar{\lambda}^2 \Gamma_1 + \bar{\lambda}\Gamma_0+\Gamma_1^T) =0.
 \]
Then the determinant of its conjugate transpose must be 0 as well, and thus
 \[
  \det (\lambda^2 \Gamma_1^T + \lambda\Gamma_0+\Gamma_1) =0.
 \]
Multiply everything by $\frac{1}{\lambda^2}$, to obtain
\[
  0=\det \frac{1}{\lambda^2}(\lambda^2 \Gamma_1^T + \lambda\Gamma_0+\Gamma_1) = \det \left(\Gamma_1^T+\frac{1}{\lambda}\Gamma_0 + \left(\frac{1}{\lambda}\right)^2\Gamma_1\right).
\]
Since this determinant is zero, the matrix is singular and there must be a vector $\tilde{u}$ such that
\[
 \left(\Gamma_1^T+\frac{1}{\lambda}\Gamma_0 + \left(\frac{1}{\lambda}\right)^2\Gamma_1\right)\tilde{u}=0.
\]
But this equation shows that the pair $\left(\frac{1}{\lambda},\tilde{u}\right)$ is also a generalized eigenpair of the polynomial eigenvalue problem.
\end{proof}
\subsection*{Proof of Lemma~\ref{eigenlemma}}
\begin{proof}
 \begin{enumerate}
  \item Let $\lambda,u$ be a solution to \eqref{pep}.
We may verify explicitly that the equation
\[
\begin{bmatrix}
 0 & I\\
 -\Gamma_1^{-1}\Gamma_1^T & -\Gamma_1^{-1}\Gamma_0
\end{bmatrix}
\begin{bmatrix}
 u\\ \lambda u
\end{bmatrix}
=\lambda
\begin{bmatrix}
 u\\ \lambda u
\end{bmatrix}
\]

holds.
\item Let $w$ be an eigenvector of $P$ with eigenvalue $\lambda$, and partition it as
\[
 w=\begin{bmatrix}
    u\\v
   \end{bmatrix}.
\]
From the eigenvalue condition $Pw=\lambda w$ we obtain
\[
\begin{bmatrix}
 0 & I\\
 -\Gamma_1^{-1}\Gamma_1^T & -\Gamma_1^{-1}\Gamma_0
\end{bmatrix}
\begin{bmatrix}
 u\\ v
\end{bmatrix}
=\lambda
\begin{bmatrix}
 u\\ v
\end{bmatrix}.
\]
Expanding the two blocks and eliminating $v$ from the resulting equations, one gets \eqref{pep}.
\end{enumerate}
\end{proof}
\subsection*{Proof of Lemma~\ref{analytical}}
The proof follows from some classical results in matrix polynomials that can be found, for instance, in \cite{GohLR}. We give the sketch of a self-contained proof here. We start from a classical result in complex analysis, the Cauchy integral formula 
\[
 \frac{1}{2\pi i} \int_{\abs{z}=1} (z  - \lambda)^{-1}\,\textrm{d}z = \begin{cases}1 & \abs{\lambda}<1, \\ 0 & \abs{\lambda}>1. \end{cases}
\]
From this, a matrix version of the same integral follows for diagonal matrices
\[
 \frac{1}{2\pi i} \int_{\abs{z}=1} (z I - D)^{-1}\,\textrm{d}z = \Pi,
\]
where $\Pi$ is the diagonal matrix such that $\Pi_{ii}$ is zero if $\abs{D_{ii}}>1$ and one if $\abs{D_{ii}}<1$. Now a change of bases in both sides of the equation yields for all diagonalizable $A$ without unimodular eigenvalues
\begin{equation}\label{intformula}
 \frac{1}{2\pi i} \int_{\abs{z}=1} (z I - A)^{-1}\,\textrm{d}z = \Pi_A,
\end{equation}
with $\Pi_A$ the projector on the invariant subspace of $A$ associated to the eigenvalues inside the unit circle. 

We may generalize further this formula to all $A$ without unimodular eigenvalues, removing the diagonalizability of $A$ from the requirements. Indeed, for a non-diagonalizable $A$, let us consider a sequence of matrices $A_k$, each of them diagonalizable, that converge uniformly to $A$. Such a sequence exists because diagonalizable matrices are dense in $\mathbb{C}^{n\times n}$. Since the integrand function is olomorphic on the integration contour, limits and derivatives can be moved inside the integral. This shows that $\Pi_A$ is an analytical function of $A$.

In particular, we apply the formula for $A$ equal to the $2\bar{d}\times 2\bar{d}$ matrix $P$ in \eqref{defP}, for which $\Pi_P$ has rank $\bar{d}$ due to the eigenvalue pairing. Let us take any $2\bar{d}\times \bar{d}$ matrix $V$ such that $\Pi_P V$ has full rank and spans the range $\im \Pi_P$. Since $\m{I\\B^T}$ is another basis for the same subspace, it follows that
\begin{equation}\label{BTformula}
B^T=\left(\m{0 & I_{\bar{d}}}\Pi_P V\right)\left(\m{I_{\bar{d}} & 0}\Pi_P V\right)^{-1}.
\end{equation}
Since invertibility is a condition that holds on an open domain, \eqref{BTformula} holds locally with a constant $V$ and provides an analytical expression for $B$ in terms of $\Pi_P$ and thus of $P$. The matrix $P$ is in turn a function of $\Gamma_0$, $\Gamma_1$.

\section{Expressions for the derivatives}\label{appDer}

In this section, we give a computable form for the Jacobian of $F: (h,M_0,M_1,M_2) \mapsto (c,A,B)$, the function considered in Theorem~\ref{Normality}. Rather than using vectorization to give an unwieldy matrix expression, we focus on describing its action as a linear map that takes a first-order perturbation of the moments (denoted by $(\dot{h},\dot{M}_0,\dot{M}_1,\dot{M}_2)$) to one of the parameters $(\dot{c},\dot{A},\dot{B})$. We shall use several times the expression for the derivative of the matrix inverse $\frac{d}{dt}(M^{-1}) = -M^{-1}(\frac{d}{dt}M) M^{-1}$.

The relation between $(\dot{h},\dot{M}_0,\dot{M}_1,\dot{M}_2)$ and $\dot{\Gamma}_0,\dot{\Gamma}_1$ is easy to compute, by simply differentiating \eqref{Gammas}:
\begin{equation}\label{der1}
\begin{aligned}
 \dot{\Gamma}_1 &= \dot{M}_1 - \dot{\Phi} M_0 - \Phi \dot{M}_0,\\
 \dot{\Gamma}_0 &= \dot{M}_0- \dot{M}_1\Phi^T-M_1 \dot{\Phi}^T - \dot{\Phi} M_1^T - \Phi \dot{M}_1^T + \dot{\Phi} M_0 \Phi^T + \Phi \dot{M}_0\Phi^T + \Phi M_0 \dot{\Phi}^T,
\end{aligned}
\end{equation}

with $\dot{\Phi} = \dot{M}_2 M_1^{-1} - M_2 M_1^{-1}\dot{M_1}M_1^{-1} = \dot{M}_2 M_1^{-1} - \Phi\dot{M_1}M_1^{-1}$, obtained by differentiating \eqref{Phi} for $k=1$.

We now differentiate \eqref{nme} to obtain
\begin{equation}\label{der2}
 \dot{\Sigma} - B \dot{\Sigma}B^T  = \dot{\Sigma} - \Gamma_1 \Sigma^{-1} \dot{\Sigma}\Sigma^{-1}\Gamma_1^T = \dot{\Gamma_0} - \dot{\Gamma}_1\Sigma^{-1}\Gamma_1^T - \Gamma_1 \Sigma^{-1}\dot{\Gamma}_1^T = \dot{\Gamma_0} + \dot{\Gamma}_1 B^T + B\dot{\Gamma}_1^T.
\end{equation}

This is a discrete-time Lyapunov equation (see for instance \cite{Lyap}) for $\dot{\Sigma}$, which can be solved in closed form by vectorization or numerically by procedures such as Matlab's \texttt{dlyap}. The equation is uniquely solvable since we are assuming that $\rho(B)<1$.

Once we have $\dot{\Sigma}$, we differentiate $B=-\Gamma_1 \Sigma^{-1}$ to obtain 
\begin{equation}\label{der3}
 \dot{B}=-\dot{\Gamma}_1 \Sigma^{-1} + \Gamma_1 \Sigma^{-1}\dot{\Sigma}\Sigma^{-1} = -(\dot{\Gamma}_1+B\dot{\Sigma}) \Sigma^{-1}.
\end{equation}

The derivatives of the remaining two parameters are given by $\dot{A}=\dot{\Phi} - \dot{B}$ and $\dot{c}=(I-\dot{\Phi})h + (I-\Phi)\dot{h}$.

Putting together \eqref{der1}, \eqref{der2}, \eqref{der3}, one can get to an expression for $(\dot{c},\dot{A},\dot{B})$ as a function of $(\dot{h},\dot{M}_0,\dot{M}_1,\dot{M}_2)$. It does not look like there are any significant simplifications in the resulting expressions. The main message to infer from this computation is that the norm of $(I-B\otimes B)^{-1}$, which appears when solving the discrete-time Lyapunov equation, has an impact on the magnitude of the derivatives; the closer $B$ is to having unimodular eigenvalues, the more ill-conditioned the solution becomes. 

\bibliographystyle{plainnat}
\bibliography{estimator}

\begin{thebibliography}{21}
\providecommand{\natexlab}[1]{#1}
\providecommand{\url}[1]{\texttt{#1}}
\expandafter\ifx\csname urlstyle\endcsname\relax
  \providecommand{\doi}[1]{doi: #1}\else
  \providecommand{\doi}{doi: \begingroup \urlstyle{rm}\Url}\fi

\bibitem[Bauwens et~al.(2006)Bauwens, Laurent, and Rombouts]{BLR}
Luc Bauwens, S{\'e}bastien Laurent, and Jeroen V.~K. Rombouts.
\newblock Multivariate {GARCH} models: a survey.
\newblock \emph{Journal of Applied Econometrics}, 21\penalty0 (1):\penalty0
  79--109, 2006.

\bibitem[Bhatia(2007)]{Bhatia}
Rajendra Bhatia.
\newblock \emph{Positive definite matrices}.
\newblock Princeton Series in Applied Mathematics. Princeton University Press,
  Princeton, NJ, 2007.
\newblock ISBN 978-0-691-12918-1; 0-691-12918-5.

\bibitem[Bollerslev(1990)]{B}
Tim Bollerslev.
\newblock Modeling the coherence in short-run nominal exchange rates: A
  multivariate generalized {ARCH} approach.
\newblock \emph{Review of Economics and Statistics}, 72:\penalty0 498--505,
  1990.

\bibitem[Bollerslev et~al.(1988)Bollerslev, Engle, and Wooldridge]{BEW}
Tim Bollerslev, Robert~F Engle, and Jeffrey~M Wooldridge.
\newblock A capital asset pricing model with time-varying covariances.
\newblock \emph{Journal of Political Economy}, 96\penalty0 (1):\penalty0
  116--31, February 1988.
\newblock URL \url{http://ideas.repec.org/a/ucp/jpolec/v96y1988i1p116-31.html}.

\bibitem[Boussama(2006)]{Bou06}
Farid Boussama.
\newblock Ergodicit\'e des cha\^\i nes de {M}arkov \`a valeurs dans une
  vari\'et\'e alg\'ebrique: application aux mod\`eles {GARCH} multivari\'es.
\newblock \emph{C. R. Math. Acad. Sci. Paris}, 343\penalty0 (4):\penalty0
  275--278, 2006.
\newblock ISSN 1631-073X.
\newblock \doi{10.1016/j.crma.2006.06.027}.
\newblock URL \url{http://dx.doi.org/10.1016/j.crma.2006.06.027}.

\bibitem[Chr{\'e}tien and Ortega(2012)]{CO}
St{\'e}phane Chr{\'e}tien and Juan-Pablo Ortega.
\newblock Multivariate {GARCH} estimation via a {B}regman-proximal trust-region
  method.
\newblock \emph{Computational Statistics and Data Analysis}, 2012.
\newblock To appear. Available as arXiv:1101.5475.

\bibitem[Comte and Lieberman(2003)]{CL}
F.~Comte and O.~Lieberman.
\newblock Asymptotic theory for multivariate {GARCH} processes.
\newblock \emph{Journal of Multivariate Analysis}, 84\penalty0 (1):\penalty0 61
  -- 84, 2003.
\newblock ISSN 0047-259X.
\newblock \doi{10.1016/S0047-259X(02)00009-X}.
\newblock URL
  \url{http://www.sciencedirect.com/science/article/pii/S0047259X0200009X}.

\bibitem[Engle and Kroner(1995)]{EK}
Robert~F. Engle and Kenneth~F. Kroner.
\newblock Multivariate simultaneous generalized {ARCH}.
\newblock \emph{Econometric Theory}, 11\penalty0 (1):\penalty0 122--150, 1995.
\newblock ISSN 0266-4666.
\newblock \doi{10.1017/S0266466600009063}.
\newblock URL \url{http://dx.doi.org/10.1017/S0266466600009063}.

\bibitem[Engwerda et~al.(1993)Engwerda, Ran, and Rijkeboer]{EngRR93}
Jacob~C. Engwerda, Andr{\'e} C.~M. Ran, and Arie~L. Rijkeboer.
\newblock {N}ecessary and sufficient conditions for the existence of a positive
  definite solution of the matrix equation {$X+A^*X^{-1}A=Q$}.
\newblock \emph{Linear Algebra Appl.}, 186:\penalty0 255--275, 1993.
\newblock ISSN 0024-3795.
\newblock \doi{10.1016/0024-3795(93)90295-Y}.
\newblock URL \url{http://dx.doi.org/10.1016/0024-3795(93)90295-Y}.

\bibitem[Francq and Zako\"ian(2010)]{FZ}
C.~Francq and J.M. Zako\"ian.
\newblock \emph{{GARCH} Models: Structure, Statistical Inference and Financial
  Applications}.
\newblock Wiley, 2010.
\newblock ISBN 9780470683910.

\bibitem[Gajic and Qureshi(1995)]{Lyap}
Zoran Gajic and Muhammad Tahir~Javed Qureshi.
\newblock \emph{{L}yapunov matrix equation in system stability and control},
  volume 195 of \emph{Mathematics in Science and Engineering}.
\newblock Academic Press Inc., San Diego, CA, 1995.
\newblock ISBN 0-12-273370-3.

\bibitem[Gohberg et~al.(1982)Gohberg, Lancaster, and Rodman]{GohLR}
I.~Gohberg, P.~Lancaster, and L.~Rodman.
\newblock \emph{{M}atrix polynomials}.
\newblock Academic Press Inc. [Harcourt Brace Jovanovich Publishers], New York,
  1982.
\newblock ISBN 0-12-287160-X.
\newblock Computer Science and Applied Mathematics.

\bibitem[Gouri{\'e}roux(1997)]{G}
Christian Gouri{\'e}roux.
\newblock \emph{A{RCH} models and financial applications}.
\newblock Springer Series in Statistics. Springer-Verlag, New York, 1997.
\newblock ISBN 0-387-94876-7.
\newblock \doi{10.1007/978-1-4612-1860-9}.
\newblock URL \url{http://dx.doi.org/10.1007/978-1-4612-1860-9}.

\bibitem[Hafner(2003)]{H1}
Christian~M. Hafner.
\newblock Fourth moment structure of multivariate {GARCH} models.
\newblock \emph{Journal of Financial Econometrics}, 1\penalty0 (1):\penalty0
  26--54, 2003.
\newblock \doi{10.1093/jjfinec/nbg001}.
\newblock URL \url{http://jfec.oxfordjournals.org/content/1/1/26.abstract}.

\bibitem[Hafner(2008)]{H2}
Christian~M. Hafner.
\newblock Temporal aggregation of multivariate {GARCH} processes.
\newblock \emph{J. Econometrics}, 142\penalty0 (1):\penalty0 467--483, 2008.
\newblock ISSN 0304-4076.
\newblock \doi{10.1016/j.jeconom.2007.08.001}.
\newblock URL \url{http://dx.doi.org/10.1016/j.jeconom.2007.08.001}.

\bibitem[Hafner and Rombouts(2007)]{HR}
Christian~M. Hafner and Jeroen V.~K. Rombouts.
\newblock Estimation of temporally aggregated multivariate {GARCH} models.
\newblock \emph{J. Stat. Comput. Simul.}, 77\penalty0 (7-8):\penalty0 629--650,
  2007.
\newblock ISSN 0094-9655.
\newblock \doi{10.1080/10629360600616252}.
\newblock URL \url{http://dx.doi.org/10.1080/10629360600616252}.

\bibitem[Jeantheau(1998)]{J}
Thierry Jeantheau.
\newblock Strong consistency of estimators for multivariate {ARCH} models.
\newblock \emph{Econometric Theory}, 14\penalty0 (1):\penalty0 70--86, 1998.
\newblock ISSN 0266-4666.
\newblock \doi{10.1017/S0266466698141038}.
\newblock URL \url{http://dx.doi.org/10.1017/S0266466698141038}.

\bibitem[Kristensen and Linton(2006)]{LK}
Dennis Kristensen and Oliver Linton.
\newblock A closed-form estimator for the {GARCH} {$(1,1)$} model.
\newblock \emph{Econometric Theory}, 22\penalty0 (2):\penalty0 323--337, 2006.
\newblock ISSN 0266-4666.
\newblock \doi{10.1017/S0266466606060142}.
\newblock URL \url{http://dx.doi.org/10.1017/S0266466606060142}.

\bibitem[Mackey et~al.(2006)Mackey, Mackey, Mehl, and Mehrmann]{Vibes}
D.~Steven Mackey, Niloufer Mackey, Christian Mehl, and Volker Mehrmann.
\newblock Structured polynomial eigenvalue problems: good vibrations from good
  linearizations.
\newblock \emph{SIAM J. Matrix Anal. Appl.}, 28\penalty0 (4):\penalty0
  1029--1051 (electronic), 2006.
\newblock ISSN 0895-4798.
\newblock \doi{10.1137/050628362}.
\newblock URL \url{http://dx.doi.org/10.1137/050628362}.

\bibitem[Meini(2002)]{Mei02}
Beatrice Meini.
\newblock Efficient computation of the extreme solutions of {$X+A^*X^{-1}A=Q$}
  and {$X-A^*X^{-1}A=Q$}.
\newblock \emph{Math. Comp.}, 71\penalty0 (239):\penalty0 1189--1204
  (electronic), 2002.
\newblock ISSN 0025-5718.
\newblock \doi{10.1090/S0025-5718-01-01368-0}.
\newblock URL \url{http://dx.doi.org/10.1090/S0025-5718-01-01368-0}.

\bibitem[Reinsel(1997)]{R}
Gregory~C. Reinsel.
\newblock \emph{Elements of multivariate time series analysis}.
\newblock Springer Series in Statistics. Springer-Verlag, New York, second
  edition, 1997.
\newblock ISBN 0-387-94918-6.

\end{thebibliography}

\end{document}